\documentclass[a4paper,11pt]{article}

\usepackage{amssymb, amsmath, color}

\def\R{\mathbb{R}}
\def\N{\mathbb{N}}

\def\Q{\mathbb{Q}}
\def\Z{\mathbb{Z}}

\def\bm{\boldsymbol}
\def\sg {{\rm sg}}
\def\mult {{\rm mult}}
\def\Dom {{\rm Dom}}

\newtheorem{definition}{Definition}

\newtheorem{lemma}[definition]{Lemma}
\newtheorem{proposition}[definition]{Proposition}
\newtheorem{theorem}[definition]{Theorem}
\newtheorem{corollary}[definition]{Corollary}
\newtheorem{remark}[definition]{Remark}
\newtheorem{notation}[definition]{Notation}
\newtheorem{example}[definition]{Example}

           {\vspace{3.3mm}
           \noindent{\bf #1}\it}%
           {\vspace{3.3mm}}

\newenvironment{proof}[1]{
  \trivlist \item[\hskip \labelsep{\it #1}]}{\hfill\mbox{$\square$}
  \endtrivlist}

\newcommand{\TaQ}{{\rm{TaQ}}}
\newcommand{\pder}{{\rm pder}}
\newcommand{\pdeg}{{\rm pdeg}}

\title{ Decision problem for a class of univariate Pfaffian functions}
\author{Mar\'\i a Laura Barbagallo$^{\natural, \flat,}$\footnote{Partially supported by the following grants: PIP 11220130100527CO (CONICET) and UBACYT 20020160100039BA (2017-2019).}, Gabriela Jeronimo$^{{ \natural,\flat,\diamondsuit},*}$, Juan Sabia$^{{\flat, \diamondsuit}, * ,}$\footnote{Corresponding author. e-mail: jsabia@dm.uba.ar}
\\[3mm]
{\small ${\natural}$ Departamento de Matem\'atica, FCEN, Universidad de Buenos Aires, Argentina}\\
{\small $\flat$ Departamento de Ciencias Exactas, CBC, Universidad de Buenos Aires, Argentina}\\
{\small ${\diamondsuit}$ IMAS, CONICET--UBA, Argentina}\\
}

\begin{document}

\maketitle

\begin{abstract}
We address the decision problem for sentences involving univariate functions constructed from a fixed Pfaffian function of order $1$. We present a new symbolic procedure solving this problem with a computable complexity based on the computation of suitable Sturm sequences. For a general Pfaffian function, we assume the existence of an oracle to determine the sign that a function of the class takes at a real algebraic number. For E-polynomials, we give an effective algorithm solving the problem without using oracles and apply it to solve a similar decision problem in the multivariate setting. Finally, we introduce a notion of Thom encoding for zeros of an E-polynomial and describe an algorithm for their computation.
\end{abstract}


\section{Introduction}

Pfaffian functions are analytic functions which are solutions of triangular systems of first order partial differential equations with polynomial coefficients. This class of functions, introduced by Khovaskii in \cite{Kho80}, includes polynomials, exponentials, logarithms and trigonometric functions in bounded intervals, among others, and satisfies global finiteness properties similar to polynomials.  For example, a system of $n$ equations given by Pfaffian functions in $n$ variables defined in a domain in $\R^n$ has finitely many non-degenerate solutions (see \cite{Kho91}). This behavior allows to prove effective and algorithmic results such as bounds on the complexity of basic operations with Pfaffian functions and sets defined by them (see, for example,  \cite{GV04}).

For the case of polynomials over $\R$, in \cite{Tarski51}, Tarski proved a quantifier elimination method for the first order theory of the real numbers.  He also posed the same question for this theory extended with exponentials, which are a particular subclass of Pfaffian functions. In \cite{vdD84} this question was answered negatively and, later, in \cite{MW96} the decidability of this extended theory was proved, provided Schanuel's conjecture is true,  using a model-theoretic approach not suitable for implementation. Afterwards, the decision problem and some related questions were considered for fragments of the first order theory of the reals extended with a particular Pfaffian function from an algorithmic viewpoint (see, for example, \cite{Richardson91}, \cite{Vor92}, \cite{Maignan98}, \cite{AW00}, \cite{Weis00} and \cite{AMW08}).

Recently, in \cite{MW12}, a decision procedure for a certain class of first order sentences involving integral polynomials and a certain specific analytic transcendental function was given and, in \cite{XLY15}, a quantifier elimination method for formulas involving an exponential function in a particular variable was proposed. Both these methods rely on the isolation of real zeros of univariate transcendental functions by means of a bisection based algorithm; their theoretical complexity has not been analyzed and, moreover, obtaining complexity estimates seems to be a difficult task.

In this paper, we address the decision problem for first order sentences constructed from atomic formulas of the type $f(x) >0$, $f(x)=0$ and $f(x)<0$, where $f(x) = F(x,\varphi(x))$ with $F\in \Z[X,Y]$, for a fixed univariate Pfaffian function $\varphi$ of order $1$ (see Section \ref{subsec:functionclass} for the definition of this class of functions).  We present a symbolic algorithm solving this problem with computable complexity. The procedure relies on the computation of suitable Sturm sequences we introduce in this framework, generalizing the ones used in \cite{BJS16} for zero counting. For general Pfaffian functions of order $1$, in order to determine their sign at a real  algebraic number, we assume the existence of an oracle as it is usual in the literature. Our main result is the following:

\begin{theorem}\label{thm:Pfaffiandecision}
Let $\varphi$ be a Pfaffian function satisfying  $\varphi'(x) = \Phi(x, \varphi (x))$ for $\Phi \in \Z[X,Y]$ with $\deg (\Phi )\le \delta$.
Let $\Psi$ be a quantifier-free formula  in a variable  $x$ involving  functions  $g_1, \dots, g_s$ defined as $g_i(x) = G_i(x, \varphi(x))$ for $G_i \in \Z[X,Y]$ with   $\deg G_i \le d$, for $1\le i \le s$. There is a symbolic procedure that determines the truth value of the formula $\exists x \, \Psi$  in an interval $[\alpha, \beta] \subseteq \hbox{Dom}(\varphi)$, for $\alpha<\beta$ real algebraic numbers,  within complexity $O(s\delta(sd+\delta)^{13}  \log^3(sd+\delta) + \delta (sd +\delta)^4 |\Psi| )$, where $|\Psi|$ denotes the length of the formula $\Psi$.
\end{theorem}

For the particular case of univariate E-polynomials, that is, when $\varphi(x) = e^{h(x)}$ for an integral polynomial $h$, from the procedure underlying the above theorem, we obtain an oracle-free symbolic algorithm that solves the decision problem  by using a subroutine to determine the sign of an E-polynomial at a real algebraic number presented in \cite{BJS16}.

Finally, we apply the algorithm developed for univariate E-polynomials in two different situations. In the setting of \emph{multivariate} E-polynomials, we construct a symbolic algorithm with bounded complexity for the decision problem for prenex formulas with one block of existential quantifiers.
On the other hand, we introduce a notion of Thom encoding for real zeros of univariate E-polynomials that generalizes the one used for real algebraic numbers and allows us to deal symbolically with the different zeros of an E-polynomial. We give an algorithm to compute these encodings and estimate the complexity.

The paper is organized as follows: Section \ref{sec:preliminaries} introduces the notation we use throughout the paper, the basic complexity results on which our algorithms rely, and the class of univariate Pfaffian functions of order $1$ we will consider. In Section \ref{sec:Sturm}, we construct suitable Sturm sequences associated to Pfaffian functions of this class and prove their main properties, from the theoretical and algorithmic point of view. Section \ref{sec:decision} deals with the decision problem for formulas constructed from univariate Pfaffian functions of order $1$. Finally, Section \ref{sec:Epolynomials} is devoted to E-polynomials: first, we prove our main complexity results for the decision problem and then, we focus on Thom encodings in this setting.

\section{Preliminaries}\label{sec:preliminaries}

\subsection{Basic notation and results}

Throughout this paper, we will deal with polynomials, mainly univariate or bivariate, with integer coefficients. For a polynomial
$F\in \Z[X,Y]$, we write $\deg_X(F)$ and $\deg_Y(F)$ for the degrees of $F$ in the variables $X$ and $Y$ respectively, $\deg(F)$ for its total degree, and $H(F)$ for its height, that is, the maximum of the absolute values of the coefficients of $F$.

For  $P_1, P_2\in \Z[X]$, we have that
$H(P_1P_2) \le (\min\{\deg(P_1), \deg(P_2)\} +1) H(P_1) H(P_2)$.
We can also estimate the height of the product of bivariate integer polynomials as follows: if $P_1,\dots, P_m \in \Z[X.Y]$, then
\begin{equation} \label{eq:productheight}
H(\prod_{i=1}^{m} P_i) \le \prod_{i=1}^{m} H(P_i) (\deg(P_i) +1)^2
\end{equation}
(see, for instance, \cite[Remark 8.24]{BPR}).

For $\gamma= (\gamma_0,\dots, \gamma_N)\in \R^{N+1}$ with $\gamma_i \ne 0$ for every $0\le i \le N$, the \emph{number of variations in sign} of $\gamma$ is the cardinality of the set $\{1\le i\le N \mid \gamma_{i-1} \gamma_i <0\}$. For a tuple $\gamma$ of arbitrary real numbers, the number of variations in sign of $\gamma$ is the number of variations in sign of the tuple which is obtained from $\gamma$ by removing its zero coordinates. Given $c\in \mathbb{R}$ and a sequence of univariate real functions $\mathbf{f} =(f_0,\dots, f_N)$ defined at $c$,  we write $v(\mathbf{f},c)$ for the number of variations in sign of the $(N+1)-$tuple $(f_0(c),\ldots,f_N(c))$.

\subsection{Complexities} \label{sec:complexity}

The main objects our algorithms deal with are polynomials with rational coefficients and bounded degree represented by the array of all their coefficients in a pre-fixed order of monomials. The notion of complexity of an algorithm we adopt is the number of operations and comparisons between elements in $\Q$.

In our complexity bounds we will use the following complexity estimates (see \cite{vzGG}):
\begin{itemize}\itemsep=0pt
\item For a matrix in $\Q^{n\times n}$, its determinant can be obtained within complexity $O(n^\omega)$, where $\omega<2.376$ (see \cite[Chapter 12]{vzGG}).
\item The product of two polynomials in $\Q[X]$ of degrees bounded by $d$ can be done within complexity $O(M(d))$, where $M(d) := d \log(d) \log\log(d)$.
\item Interpolation of a degree $d$ polynomial in $\Q[X]$ can be achieved within $O(M(d)\log(d))$ arithmetic operations.
\end{itemize}

A basic subroutine in our algorithms will be the computation of subresultants.  We will compute them by  means of matrix determinants according to \cite[Notation 8.55]{BPR}, which enables us to control both the complexity and the
output size (an alternative method for the computation of subresultants, based on the Euclidean algorithm, can be found in \cite[Algorithm 8.21]{BPR}).

Let $P$ and $Q \in \Z[X][Y]$ be polynomials such that $\deg_X(P) = D_X$, $\deg_Y(P) = D_Y$, $\deg_X(Q) = d_X$ and $\deg_Y(Q) = d_Y$, with $D_Y > d_Y$. For $0\le j \le d_Y$, let $ \textrm{SRes}_j(P,Q)$ be the $j$th subresultant of $P$ and $Q$ (considered as polynomials in the variable $Y$), which is a polynomial of degree bounded by $j\le d_Y$ in the variable $Y$, whose coefficients are polynomials of degree bounded by $D := d_X D_Y + d_Y D_X$ in the variable $X$.

We compute all the subresultants of $P$ and $Q$ by means of interpolation: for $D+1$ interpolation points, we evaluate the coefficients of $P$ and $Q$, we compute the corresponding determinant polynomials  (which are  polynomials in $Y$ with constant coefficients) and, finally, we interpolate to obtain each coefficient of each $\textrm{SRes}_j(P,Q)\in \Z[X][Y]$.

For each interpolation point, the evaluation of the coefficients of $P$ and $Q$ can be performed within complexity $O(D_X D_Y + d_X d_Y)$.
Then, for each $0\le j\le d_Y$, in order to obtain the specialization of $\textrm{SRes}_j(P,Q)$ at a fixed interpolation point, by expanding the polynomial determinant by its last column, we compute at most $D_Y+d_Y$ determinants of matrices of size bounded by $D_Y +d_Y$,
 multiply them by the polynomials $Y^kP$ or $Y^kQ$ evaluated at the point, and add the results, within complexity $O(((D_Y+d_Y)^\omega +D_Y)d_Y +((D_Y+d_Y)^\omega +d_Y)D_Y) = O((D_Y+d_Y)^{\omega +1})$.
 Thus, the complexity to obtain all subresultants specialized at all the interpolation points is $O(D d_Y(D_Y+d_Y)^{\omega +1})$.
Finally, for every $0\le j\le d_Y$, we recover the coefficients of the subresultant polynomial $\textrm{SRes}_j(P,Q)$  by interpolation from the previous results. Since each polynomial interpolation can be done within complexity $O(M(D) \log(D))$, the computation of the at most $d_Y$ coefficients of each subresultant can be achieved within complexity $O(d_Y M(D) \log(D))$.

Therefore, the complexity to obtain all the polynomial coefficients of all subresultants is of order
$O( D(D_X D_Y + d_X d_Y) + D d_Y (D_Y+d_Y)^{\omega +1} + d_Y^2  M(D) \log(D))$.

\bigskip

We will also apply an effective procedure for the determination of all realizable sign conditions for a finite family of real univariate polynomials. A \emph{realizable sign condition for $q_1,\dots, q_s\in \R[X]$ on a finite set $Z\subset \R$} is an $s$-tuple $(\sigma_1,\dots, \sigma_s) \in \{<,=,>\}^s$ such that $ \{ x\in Z\mid q_1(x) \sigma_1 0,\dots, q_s(x) \sigma_s 0\}\ne \emptyset$. Given $p,q_1,\dots,q_s\in \R[X]$, $p\not \equiv 0$, with $\deg p, \deg q_i \le d $ for all $1\le i \le d$, the list of all realizable sign conditions for $q_1,\dots, q_s$ on the set of real roots of $p$ can be computed within $O(smd \log(m) \log^2(d)) = O(sd^2\log^3(d))$ arithmetic operations, where $m$ is the number of real roots of $p$ (see \cite[Corollary 2]{Perrucci11}).

\bigskip

 To work with real algebraic numbers in a symbolic way, we will use their \emph{Thom encodings}. For a real root $\alpha$ of a given polynomial $p\in \R[x]$, the Thom encoding of $\alpha$ as a root of $p$ is the sequence $(\sg(p'(\alpha)),\dots, \sg(p^{(\deg(p))}(\alpha))$, where $\sg$ stands for the sign of a real number, which is represented by an element of the set $\{0,1,-1\}$. If $\deg(p) = d$, given the Thom encodings of $m$ real roots $\alpha_1,\dots, \alpha_m$ of $p$, we can order them as $\alpha_{i_1}<\dots<\alpha_{i_m}$ (see \cite[Proposition 2.37]{BPR}) within complexity $O(dm\log m)$.

\subsection{A class of univariate Pfaffian functions}\label{subsec:functionclass}

We will deal with the particular class of Pfaffian functions of order $1$ we introduce in this section (for a general definition of Pfaffian functions, see \cite{GV04}).

Let $\varphi$ be a function satisfying a differential equation of the type
\begin{equation}\label{eq:BasicPfaffianFunction}
 \varphi'(x) = \Phi(x, \varphi(x))
\end{equation}
where $\Phi\in \Z[X,Y]$ and $\deg_Y(\Phi)>0$. Given a polynomial $F\in \Z[X,Y]$ with $\deg_Y(F) >0$, we define a function
\begin{equation}\label{eq:PfaffianFunction}
 f(x) = F(x, \varphi(x)).
\end{equation}
We will work with the subclass of all functions defined in this way from a fixed function $\varphi$ and we will call any function in this class \emph{a Pfaffian function associated to $\varphi$}. 

Note that, if $f$ is a Pfaffian function associated to $\varphi$, its successive derivatives are also Pfaffian functions associated to $\varphi$, since
\[ f'(x) = \frac{\partial F}{\partial X}(x, \varphi(x)) + \frac{\partial F}{\partial Y}(x, \varphi(x)). \Phi(x, \varphi(x)) = \widetilde F(x,\varphi(x)),\]
where $\widetilde F\in \Z[X,Y]$ is defined as:
\begin{equation}\label{eq:Ftilde}
\widetilde F(X,Y) = \dfrac{\partial F}{\partial X}(X,Y) + \dfrac{\partial F}{\partial Y}(X,Y)  \Phi(X,Y).
\end{equation}
For a positive integer $\nu$, we will write $\widetilde F^{(\nu)}$ for the polynomial in $\Z[X,Y]$ defining the $\nu$th-order derivative of $f$, that is, $f^{(\nu)}(x) = \widetilde F^{(\nu)}(x, \varphi(x))$.

\bigskip
A particular subclass of these functions we will consider is the one of \emph{E-polynomials}, namely, when $\varphi(x) = e^{h(x)}$ for $h\in \Z[X]$. This function  satisfies Equation \eqref{eq:BasicPfaffianFunction} for $\Phi(X,Y) =  h'(X) Y$.

\section{Sturm sequences and Tarski-queries}\label{sec:Sturm}

\subsection{Definition and main properties}

We introduce a notion of Sturm sequence associated to a pair of continuous functions that extends the definition given in \cite{Heindel71} and works in a similar way as it does for a pair of real univariate polynomials.

\begin{definition}  \label{SucSturmgral}
Let $f,g:(a,b)\rightarrow\mathbb{R}$ be continuous functions. A \emph{Sturm sequence in $(a,b)$ for $f$ with respect to $g$} is a finite sequence of continuous functions $\mathbf{f}= (f_0, f_1, \ldots, f_N)$ defined in $(a,b)$ satisfying the following properties for every $y\in (a,b)$:
\begin{enumerate}
	\item $f_0(y)=0$ if and only if $f(y)= 0$ and $g(y)\neq0$.
	\item If $f_0(y)=0$, then there exists $ \epsilon >0$ such that $f_1(x)\neq 0$ for all $x\in (y-\epsilon,y+\epsilon)$. Moreover,

if $g(y)>0$, then $$\begin{cases} f_0(x)f_1(x)<0 & \hbox{ for } \ y-\epsilon<x<y \\ f_0(x)f_1(x)>0 & \hbox{ for } \ y<x<y+\epsilon,\end{cases}$$
	
and, if $g(y)<0$, then $$\begin{cases} f_0(x)f_1(x)>0 & \hbox{ for }  \ y-\epsilon<x<y \\ f_0(x)f_1(x)<0 & \hbox{ for }  \ y<x<y+\epsilon.\end{cases}$$
	\item For $i=1,\ldots,N-1$, if $f_i(y)=0$, then  $f_{i-1}(y)f_{i+1}(y)<0$.
	\item $f_N(y)\neq 0$.
\end{enumerate}
\end{definition}

If, for a given $x\in \mathbb{R}$,  $v(\mathbf{f},x)$ denotes the number of variations in sign of the $(N+1)-$uple $\mathbf{f}(x):=(f_0(x),\ldots,f_N(x))$, we have (c.f.~\cite[Theorem 2.1]{Heindel71}):

\begin{theorem}\label{resultadosturm}
Let $f, g:(a,b) \to \R$ be continuous functions and $\mathbf{f} = (f_0,\ldots, f_N)$ a Sturm sequence for $f$ with respect to $g$ in $(a,b)$. If $c,d\in (a,b)$, $c<d$, are not zeros of $f$, then $$v(\mathbf{f}, c)-v(\mathbf{f},d)=\#\left\{x\in (c,d) \mid f(x)=0 \wedge g(x)>0\right\}-\#\left\{x\in (c,d) \mid f(x)=0 \wedge g(x)<0\right\}.$$
\end{theorem}

\begin{proof}{Proof.}
Let us analyze the function $v(x) := v(\mathbf{f},x)$.

Let $t\in [c,d]\subset (a,b)$.
Since $f_0,\ldots,f_N$ are continuous, there exists $\epsilon>0$ such that for every $i$ with $f_i(t)\neq 0$, the function $f_i$ does not change sign (and, therefore, it does not vanish) in $(t-\epsilon,t+\epsilon)$. On the other hand, for all $i\neq 0, N$ such that $f_i(t)=0$,  by condition $3$ in Definition \ref{SucSturmgral}, we have that $f_{i-1}(t)f_{i+1}(t)<0$. Then, $f_{i-1}$ and $f_{i+1}$ have constant and opposite signs in $(t-\epsilon,t+\epsilon)$ for all $i\neq 0, N$ with $f_i(t)=0$.
Note that $f_N(t) \ne 0$ for every $t\in (a,b)$ because of condition 4 in Definition \ref{SucSturmgral} and so, $f_N$ does not change sign in $(a,b)$.

If $f_0(t)\neq 0$, 
we conclude that $v(x)=v(y)$ for all $ x,y\in(t-\epsilon,t+\epsilon)$, that is, 
\begin{equation}v(x)=v(t) \ \hbox{ for } x\in (t-\epsilon,t+\epsilon). \label{vcongnula}\end{equation}
	
If $f_0(t)=0$, by condition 2 in Definition \ref{SucSturmgral}, $f_1(t)\neq 0$; moreover, for $\epsilon$ sufficiently small, $f_1$ does not change sign in $(t-\epsilon,t+\epsilon)$ and,
if $a\leq t-\epsilon<x<t<y<t+\epsilon\leq b$,
we have that
\begin{itemize}
\item if $g(t)>0$, then $\sg(f_0(x))=-\sg(f_1(x))\neq 0$ and $\sg(f_0(y))=\sg(f_1(y))\neq 0$,
\item if $g(t)<0$, then $\sg(f_0(x))=\sg(f_1(x))\neq 0$ and $\sg(f_0(y))=-\sg(f_1(y))\neq 0$.
\end{itemize}
Therefore,  if $g(t)>0$, then $v(x)=v(y)+1$, and if $g(t)<0$, then $v(x)=v(y)-1$.
Moreover, since $f_0(t)=0$,
we conclude that
\begin{equation} \begin{array}{l} \hbox{if} \ g(t)>0: \ v(x)=\begin{cases}v(t)+1 &  \hbox{ for } x\in (t-\epsilon,t) \\ v(t) & \hbox{ for } x\in [t,t+\epsilon)\end{cases}\\[5mm]
\hbox{if} \ g(t)<0: \ v(x)=\begin{cases}v(t) & \hbox{ for } x\in (t-\epsilon,t]\\v(t)+1 &  \hbox{ for } x\in (t,t+\epsilon).\end{cases}\end{array}\label{vcongnonula}\end{equation}

In this way, we obtain a covering by open intervals  $(I_t)_{t\in[c,d]}$ centered at $t$ of the closed interval $[c,d]$, such that, for every $t\in [c,d]$, either $f_0(t)\neq 0$ and (\ref{vcongnula}) holds, or $f_0(t)=0$ and (\ref{vcongnonula}) holds.
By compactness, there exists $c\leq t_1<t_2<\ldots<t_k\leq d$ such that $[c,d] \subset \bigcup_{1\le i \le k} I_{t_i}$. We may assume that none of these intervals is contained in another one and, therefore, that $I_{t_i}\cap I_{t_{i+1}}\neq\emptyset$ for all $i=1,\ldots,k-1$.

Let $r_1<\ldots<r_s$ be those values $t_i$, $i=1,\dots, k$, that are zeros of $f_0$. By condition 1 of Definition \ref{SucSturmgral}, $g(r_j)\ne 0$ for every $j=1,\dots, s$. Furthermore,
\begin{itemize}
\item $c<r_1$ and $r_s<d$, since $c$ and $d$ are not zeros of $f$ and so, they are not zeros of $f_0$.
\item If $t\in (c,d)$ and $t\neq r_j$ for all $j=1,\ldots,s$, then $f_0(t)\neq 0$, since $t\in I_{t_i}$ for some $i=1,\ldots,k$, and, because of the construction of  $I_{t_i}$, $f_0(x)\ne 0$ for every $x\in I_{t_i}$, $x\ne t_i$.
\end{itemize}
As a consequence,
$$\left\{x\in(c,d) \mid f(x)=0, \, g(x)>0\right\}=\left\{x\in(c,d) \mid f_0(x)=0, \, g(x)>0\right\} = \left\{r_j \mid g(r_j)>0\right\},$$
$$\left\{x\in(c,d) \mid f(x)=0, \, g(x)<0\right\}=\left\{x\in(c,d) \mid f_0(x)=0, \, g(x)<0\right\} = \left\{r_j \mid g(r_j)<0\right\}.$$
	
On the other hand, we have that $v$ is constant in $[c, r_1)$, in $(r_j,r_{j+1})$ for every $j=1,\ldots,s-1$, and in $(r_s, d]$, since $(I_{t_i})_{i=1,\ldots,k}$ is a covering of $[c,d]$ satisfying (\ref{vcongnula}) and (\ref{vcongnonula}).

For every $j=1,\ldots,s-1$, let $\xi_j\in (r_j,r_{j+1})$, and also take $\xi_0=c$ and $\xi_{s}= d$.
Due to (\ref{vcongnula}) and (\ref{vcongnonula}), we have that for every $j=1,\dots, s$,
$$v(\xi_{j-1})-v(\xi_{j})=\begin{cases} \ 1 & \mbox{ if } g(r_{j})>0 \\  -1 & \mbox{ if } g(r_{j})<0\end{cases}.$$
Therefore,
$v(c)-v(d)=\sum_{j=1}^{s}v(\xi_{j-1})-v(\xi_{j})=\#\left\{r_j \mid g(r_j)>0\right\}-\#\left\{r_j \mid g(r_j)<0\right\}.$
\end{proof}

\subsection{Construction}

In this section we will show how to effectively compute Sturm sequences for the particular class of Pfaffian functions introduced in Section \ref{subsec:functionclass}.

Let $\varphi$ be a function satisfying a differential equation of the type
$ \varphi'(x) = \Phi(x, \varphi(x))$,
where $\Phi\in \Z[X,Y]$ and $\deg_Y(\Phi)>0$.
Given polynomials $F, G \in \Z[X,Y]$ with $\deg_Y(F)>0$ and $\deg_Y(G)>0$, we are going to work with the Pfaffian functions
\[ f(x) = F(x, \varphi(x)) \ \hbox{ and } \ g(x) = G(x, \varphi(x)). \]

According to identity \eqref{eq:Ftilde}, we may associate with $F$ a polynomial $\widetilde F \in \Z[X,Y]$ such that
\[ f'(x) = \widetilde F(x, \varphi(x)).\]

We will apply the theory of subresultants and its relation with polynomial remainder sequences presented in \cite[Section 8.3]{BPR} in order to get Sturm sequences for $f$ with respect to $g$. Similarly as in \cite{BJS16}, we will first consider a sequence of subresultant polynomials $R_i\in \Z[X,Y]$ for $i=-1,0,\dots, N$ and then, for an open interval $I$ satisfying certain assumptions, we will obtain a sequence of Pfaffian functions $\mathbf{f}_I = (f_{I,i})_{0\le i \le N}$ by substituting $Y=\varphi(x)$ in the polynomials $R_i$ and multiplying by a suitably chosen sign $\sigma_{I,i}$.

\begin{notation}\label{not:subres} Consider $F, \widetilde F$ and $G$ as polynomials in $\Z[X][Y]$. Let
\begin{itemize}\itemsep=0pt
\item $n_{-1} := \deg_Y(\widetilde F G)+1, \quad   R_{-1} := \widetilde F G $;
\item $n_0 := \deg_Y(\widetilde F G),\quad
 R_0 := F$;
\item for $i\ge 0$, if $R_i \ne 0$, let
\[n_{i+1}:= \deg_Y(R_i) \hbox{ and } R_{i+1} := {\rm{SRes}}_{n_{i+1} -1}\in \Z[X][Y]\] be the $(n_{i+1}-1)$-th subresultant polynomial associated to $\widetilde F G$ and $F$ (note that, under our assumptions, $\deg_Y(\widetilde F G) > \deg_Y(F)$).
\end{itemize}

Let $N:= \max\{i\ge 0 \mid R_i \ne 0\}$.
For $i=-1,\dots, N$, let $F_i=\dfrac{R_i}{R_N} \in  \Q(X)[Y],$
$\tau_i:=t_{n_i-1}\in \Z[X]$ be the leading coefficient of $R_i$ and, for $i=1,\dots, N+1$, let $\rho_i:=s_{n_i}\in \Z[X]$ be the $n_i$th subresultant coefficient of $\widetilde F G$ and $F$.
\end{notation}

{}From the structure theorem for subresultants (see \cite[Theorem 8.56]{BPR}), it follows that

\begin{equation}\label{eq:subres1} \epsilon \tau_0^\delta R_{-1} = C_0 R_0 + R_1 \end{equation}
where $\delta = \deg_Y(\widetilde F G) - \deg_Y (F)+1$ and $\epsilon:=(-1)^{\delta (\delta-1)/2}$ is a well defined sign,
and for $i\ge 0$,
\begin{equation}\label{eq:subres2}
\rho_{i+2} \tau_{i+1} R_i  = C_{i+1} R_{i+1} - \rho_{i+1} \tau_{i} R_{i+2}.
\end{equation}

\begin{definition}\label{def:signseq}
For an interval $I =(a,b)$  containing no root of the polynomials $\tau_i$ for $i=0,\dots, N$ or $\rho_i$ for $i=1,\dots, N+1$, we define inductively a sequence $(\sigma_{I,i})_{0\le i \le N}\in \{1,-1\}^{N+1}$  as follows:
\begin{itemize}
\item $\sigma_{I,0}=1$,
\item $\sigma_{I,1}=\epsilon \  \sg_I (\tau_0)^\delta$
\item $\sigma_{I,i+2}=  \sg_I(\rho_{i+2} \tau_{i+1} \rho_{i+1} \tau_i) \sigma_{I, i}$,
    \end{itemize}
where, for a continuous function $\theta$  of a single variable with no zeros in $I$, $\textrm{sg}_I(\theta)$ denotes the (constant) sign of $\theta$ in $I$.

If $I$ is contained in the domain of $\varphi$,  we introduce the sequence of Pfaffian functions $\mathbf{f}_I= (f_{I,i})_{0\le i\le N}$ defined by $$f_{I,i}(x) =\sigma_{I,i} F_{i}(x, \varphi(x)).$$
\end{definition}

\begin{proposition} \label{prop:sturmI} Let $\varphi$ be a Pfaffian function satisfying  $\varphi'(x) = \Phi(x, \varphi(x))$, where $\Phi \in \Z[X,Y]$ with $\deg_Y(\Phi)>0$.
Consider the functions $f(x) = F(x, \varphi(x))$ and $g(x) = G(x, \varphi(x))$, where $F, G\in \Z[X,Y]$, $\deg_Y (F)>0$, $\deg_Y(G)>0$.
With the notation and assumptions of Definition \ref{def:signseq}, the sequence of Pfaffian functions $\mathbf{f}_I= (f_{I,i})_{0\le i\le N}$  is a Sturm sequence for $f$ with respect to $g$ in $I=(a,b)$.
\end{proposition}

\begin{proof}{Proof.}  In order to shorten notation, we simply write $\sigma_{i}= \sigma_{I,i}$ and $f_i = f_{I, i}$ for $i=0,\dots, N$,  $f_{-1}(x) =  F_{-1}(x,\varphi(x))$, and $r(x) = R_N(x, \varphi(x))$.
First, let us prove that, for every $y\in I$,
\begin{equation}\label{eq:multr}
\mult(y, r) = \min\{ \mult(y, f'g), \mult(y,f)\}.
\end{equation}
By the definition of the polynomials $(R_i)_{-1\le i \le N}$, we have that
\begin{equation} \label{eq:fpg} f'(x) g(x)= R_{-1}(x,\varphi(x)) = f_{-1} (x) R_N(x, \varphi(x)) = f_{-1}(x) r(x),\end{equation}
\begin{equation} \label{eq:f} f(x) = R_{0}(x,\varphi(x)) = f_0(x) R_N(x, \varphi(x)) = f_0(x) r(x), \end{equation}
Then, for every $y\in I$, $\mult(y, r) \le \min\{\mult(y, f'g), \mult(y, f)\}$.
On the other hand, taking into account that there exist polynomials $A, B \in \Z[X, Y]$ such that $R_N = A \widetilde{F} G + B F$, it follows that
\begin{equation}\label{eq:RNcl}
r(x) = A(x, \varphi(x)) f'(x) g(x) + B(x, \varphi(x)) f(x),
\end{equation}
which implies that $\mult(y, r) \ge \min \{\mult(y, f'g), \mult(y, f)\}$.

Let us verify that all the conditions in Definition \ref{SucSturmgral} hold.
\begin{enumerate}
\item
Identity \eqref{eq:f} implies that, if $f_0(y) =0$, then $f(y) =0$ and $\mult(y, r) < \mult(y, f)$. If, in addition, $g(y)=0$, then  $\mult(y, f'g) \ge \mult (y, f)$ and so, by identity \eqref{eq:multr}, $\mult (y, r) = \mult (y, f)$, and we arrive at a contradiction. We conclude that $g(y) \ne 0$.

Conversely, if $f(y) =0$ and $g(y)\ne 0$, then $\mult(y, f'g)  =  \mult (y,f)-1$. Then, by identity \eqref{eq:multr}, $\mult(y, r) = \mult (y,f)-1$  and, from identity \eqref{eq:f}, it follows that $\mult (y, f_0)\ge 1$.

\item Let $y\in I$ be a zero of $f_0$. Then $\mult(y, f) = m>0$ and $g(y) \ne 0$, and so, $\mult(y,r) = m-1$.  If $f(x) = (x-y)^m \alpha(x)$, for an analytic function $\alpha$ defined in a neighborhood of $y$,  then $f'(x) = (x-y)^{m-1} (m \alpha(x) + (x-y) \alpha'(x))$ and
\begin{equation}\label{eq:ffminus1} f_0 (x) f_{-1}(x) = \dfrac {\alpha(x) (m \alpha(x) + (x-y) \alpha'(x)) g(x) (x-y) }{\beta^2(x)}
\end{equation}
for an analytic function $\beta$ such that $\beta(y) \ne 0$. Then  $f_{-1}(y)\ne 0$, since $g(y)\ne 0$ and $\alpha(y)\ne 0$ (so  $y$ is a zero of multiplicity $1$ of the analytic function in the left hand side of \eqref{eq:ffminus1}). By equation \eqref{eq:subres1}  and the definition of $\sigma_{I,1}$, we have that $$ (\sg_I(\tau_0))^\delta \ (\tau_0)^\delta f_{-1} = q. f_0 + f_1.$$ Therefore, $f_1(y)\ne 0$, and $f_1$ and $f_{-1}$ have the same sign in a neighborhood of $y$.

Moreover, from identity \eqref{eq:ffminus1}, we deduce that, if $g(y)>0$ then $\sg(f_0(x)f_{-1}(x)) = \sg (x-y)$ for $x$ in a neighborhood of $y$ and, if $g(y)<0$ then $\sg(f_0(x)f_{-1}(x)) = -\sg (x-y)$ for $x$ in a neighborhood of $y$.

\item Note that identity \eqref{eq:subres2} and the definition of $\mathbf{f}_I$ imply that, for $i=1,\dots, N-1$,
$$\rho_{i+1}(x) \tau_i(x) \sigma_{i-1} f_{i-1} (x) = c_i(x) \sigma_{i} f_i(x) - \rho_i(x) \tau_{i-1}(x) \sigma_{i+1} f_{i+1}(x).$$
Fix $i_0$ with $1\le i_0\le N-1$ and assume $f_{i_0}(y) = 0$.
By considering the equalities above for $i_0,i_0-1,\dots, 1$, and taking into account that the polynomials $\rho_j, \tau_j$ do not vanish in $I$, we deduce recursively that, if $f_{i_0+1}(y) = 0$ (or $f_{i_0-1}(y) = 0$), then $f_j(y) = 0$ for every $j=i_0-1,\dots,1, 0$. In particular, $f_0(y) = f_1(y)=0$, contradicting condition 2 in Definition \ref{SucSturmgral}. Therefore, $f_{i_0+1}(y)  \ne 0$ and $f_{i_0-1}(y) \ne 0$.

The fact that  $f_{i_0-1}(y) . f_{i_0+1}(y) >0$ is a direct consequence of the way in which the signs $\sigma_{i}$ for $0\le i \le N$ are defined.

\item There is nothing to prove since, by definition, $f_N$ is a non-zero constant function in $I$.
\end{enumerate}
\end{proof}

We are now going to apply the previous construction for computing Tarski queries:  given Pfaffian functions $f$ and $g$ associated to $\varphi$, the \emph{Tarski-query of $f$ for $g$ in
   $[\alpha, \beta] \subset \mbox{Dom}(\varphi)$} is  the number
$$ \TaQ(f,g; \alpha, \beta):= \#\left\{x\in (\alpha, \beta) / \ f(x)=0 \wedge g(x)>0\right\}-\#\left\{x\in (\alpha, \beta) / \ f(x)=0 \wedge g(x)<0\right\}. $$

First, we introduce some further notation.

\begin{definition} Let $\theta: [a,b]\to \mathbb{R}$ be a non-zero analytic function. For $c\in (a,b)$, we denote $\sg(\theta, c^+)$  the sign that $\theta$ takes in $(c, c+\varepsilon)$  and $\sg(\theta, c^-)$ the sign that $\theta$ takes in $(c-\varepsilon, c)$ for a sufficiently small $\varepsilon>0$. Similarly, $\sg(\theta, a^+)$ and $\sg(\theta, b^-)$ denote the signs that $\theta$ take in $(a, a+\varepsilon)$ and $(b-\varepsilon, b)$ for a sufficiently small $\varepsilon>0$.
For a sequence of non-zero analytic functions $\bm{\theta}= (\theta_0,\dots, \theta_N)$ defined in $J$, we write $v(\bm{\theta}, c^+)$ for the number of variations in sign  in $(\sg(\theta_0, c^+), \dots, \sg(\theta_N, c^+))$ and $v(\bm{\theta}, c^-)$ for the number of variations in sign in $(\sg(\theta_0, c^-), \dots, \sg(\theta_N, c^-))$.
\end{definition}

With the notation and assumptions of Definition \ref{def:signseq}, let $\mathbf{p}_I= ( p_{I,i})_{0\le i \le N}$, where
\begin{equation} \label{eq:sturmI}
p_{I,i} (x) = \sigma_{I,i} R_i(x,\varphi(x))  \quad \hbox { for } i=0,\dots,  N.
\end{equation}
Note that if $\xi \in I=(a,b)$ and $R_N(\xi, \varphi(\xi)) \ne 0$, then $v(\mathbf{p}_I, \xi)=v(\mathbf{f}_I, \xi)$. As a consequence, $v(\mathbf{p}_I, a^+)=v(\mathbf{f}_I, a^+)$ and $v(\mathbf{p}_I, b^-)=v(\mathbf{f}_I, b^-)$.

Then, by Theorem \ref{resultadosturm}, we have:

\begin{proposition}
With the previous assumptions and notation, if, in addition, the closed interval $[a,b]$ is contained in the domain of $\varphi$,
then $$v(\mathbf{p}_I, a^+)-v(\mathbf{p}_I, b^-)=\#\left\{x\in (a,b) \mid f(x)=0 \wedge g(x)>0\right\}-\#\left\{x\in (a,b) \mid f(x)=0 \wedge g(x)<0\right\}.$$
\end{proposition}

As a consequence, we obtain:

\begin{theorem}\label{thm:TarskiQuery}
Let $\varphi$ be a Pfaffian function satisfying $\varphi'(x) = \Phi(x, \varphi(x))$ for a polynomial $\Phi \in \Z[X,Y]$ with $\deg_Y(\Phi)>0$, and
let $f(x) = F(x, \varphi(x))$ and $g(x) = G(x, \varphi(x))$, where  $F, G\in \Z[X,Y]$, $\deg_Y (F)>0$, $\deg_Y (G)>0$.
Consider a bounded open interval $ (\alpha, \beta) \subset \mathbb{R}$ such that $[\alpha, \beta]$ is contained in the domain of $\varphi$.

Let $\rho_i$ and $\tau_i$ be the polynomials in $\Z[X]$ introduced in Notation \ref{not:subres}. If $\alpha_1<\alpha_2<\dots< \alpha_k$ are all the roots in $(\alpha, \beta)$  of  $\rho_i$ and $\tau_i$, then
$$ \TaQ(f,g; \alpha, \beta):= \#\left\{x\in (\alpha,\beta) \mid f(x)=0 \wedge g(x)>0\right\}-\#\left\{x\in (\alpha,\beta) \mid f(x)=0 \wedge g(x)<0\right\} =$$ $$=
 \# \{1\le j\le k \mid f(\alpha_j) = 0 \wedge g(\alpha_j)>0\} - \# \{1\le j\le k\mid f(\alpha_j) = 0 \wedge g(\alpha_j)<0\}  + {}$$ $${} + \sum_{j=0}^k v(\mathbf{p}_{I_j}, \alpha_j^+) - v(\mathbf{p}_{I_j}, \alpha_{j+1}^-),$$
where $\alpha_0 = \alpha$, $\alpha_{k+1} = \beta$ and
 for every $0\le j\le k$, $I_j = (\alpha_j, \alpha_{j+1})$ and $\mathbf{p}_{I_j}$ is the sequence of functions
  $p_{I_j,i} (x) = \sigma_{I_j,i} R_i(x,\varphi(x))$  for $i=0,\dots,  N$ with $\sigma_{I_j,i}$ introduced in Definition \ref{def:signseq} and $R_i(X,Y)$ defined in Notation \ref{not:subres}.
\end{theorem}

\subsection{Algorithmic computation}\label{subsec:TQalgorithm}

Let  $\varphi$ be a Pfaffian function satisfying  $\varphi'(x) = \Phi(x, \varphi(x))$ for a polynomial $\Phi \in \Z[X,Y]$. Let $\delta_Y:=\deg_Y(\Phi)>0$ and $\delta_X:= \deg_X(\Phi)$.

Given functions of the type $f(x) = F(x, \varphi(x))$ and $g(x) = G(x, \varphi(x))$ where $F, G\in \Z[X, Y]$ with $\deg (F) \le d$, $\deg (G) \le d$ for $d \in \Z_{\ge 0}$, $\deg_Y(F)>0$ and $\deg_Y(G)> 0$,  in this section, we describe an algorithm for computing the Tarski-query of $f$ for $g$ in $[\alpha, \beta] \subset \mbox{Dom}(\varphi)$ following Theorem \ref{thm:TarskiQuery}. To do so, we need to compute
$\sg(\theta, c^+)$   and $\sg(\theta, c^-)$ for adequate Pfaffian functions $\theta$ and real algebraic numbers $c$.

\begin{remark}\label{rem:lateralsign}
If $\theta: J \to \mathbb{R}$ is a non-zero analytic function defined in an open interval $J\subset \mathbb{R}$ and $c\in J$, then
\[ \sg(\theta, c^+) = \begin{cases} {\rm sign} (\theta(c)) & {\rm if } \ \theta(c) \ne 0\\
{\rm sign} (\theta^{(r)}(c)) &
{\rm if }\  {\rm{mult}}(c, \theta) = r
\end{cases}\]
and
\[ \sg(\theta, c^-) = \begin{cases} {\rm sign} (\theta(c)) & {\rm if } \ \theta(c) \ne 0\\
{\rm sign} ((-1)^{r} \theta^{(r)}(c)) &
{\rm if }\  {\rm{mult}}(c, \theta) = r
\end{cases}\]
where ${\rm{mult}}(c, \theta)$ is the multiplicity of $c$ as a zero of $\theta$.
\end{remark}

To estimate the complexity of the algorithm, we will need an upper bound for the multiplicity of a zero of a Pfaffian function of the considered type. We will apply the bound obtained in \cite{BJS16}:
if $\theta(x) = \Theta(x, \varphi(x))$, with $\Theta\in \Z[X, Y]$, is a nonzero Pfaffian function, for every $c\in \mathbb{R}$ such that $\theta(c)=0$,  we have
\begin{equation}\label{eq:multiplicity}\mult(c, \theta)\le 2 \deg_X(\Theta) \deg_Y(\Theta) + \deg_X(\Theta) (\delta_Y-1) + (\delta_X+1) \deg_Y(\Theta). \end{equation}

Our algorithm works as follows:

\bigskip

\noindent \hrulefill
\smallskip

\noindent\textbf{Algorithm \texttt{Tarski-query}}

\medskip

\noindent INPUT: A function $\varphi$ satisfying a differential equation $\varphi'(x) = \Phi(x, \varphi(x))$,  polynomials $F, G \in \Z[X,Y]$, and a closed interval $[\alpha,\beta]\subset\text{Dom}(\varphi)$.

\smallskip
\noindent OUTPUT: $\TaQ(f, g; \alpha, \beta)$, where $f(x) = F(x, \varphi(x))$ and $g(x) = G(x, \varphi(x))$.
\begin{enumerate}

\item Compute the polynomials $R_i$ and $\tau_i$, for $-1\le i \le N$, and $\rho_i$, for $1\le i \le N+1$,  associated to $F$ and $G$ as in Notation \ref{not:subres}.

\item Determine and order all the real roots $\alpha_1<\alpha_2<\cdots <\alpha_k$ lying in the interval $(\alpha, \beta)$ of the polynomials $\tau_i$, for $-1\le i \le N$, and $\rho_i$, for $1\le i \le N+1$.

\item For every $0\le j\le k$:
\begin{enumerate}
\item Determine the signs $\sigma_{I_j, i}$, for $0\le i \le N$, as stated in Definition \ref{def:signseq}, for $I_j:= (\alpha_j, \alpha_{j+1})$,  where $\alpha_0 = \alpha$ and $\alpha_{k+1} = \beta$.
\item Consider the sequence of functions $\mathbf{p}_{I_j}= (p_{I_j, i})_{0\le i \le N}$, with $p_{I_j, i}:= \sigma_{I_j, i}R_i(x, \varphi(x))$, introduced in \eqref{eq:sturmI}, and compute $v_j:=v(\mathbf{p}_{I_j}, \alpha_j^+) - v(\mathbf{p}_{I_j}, \alpha_{j+1}^{-})$.
\end{enumerate}

\item For every $1\le j\le k$, decide whether $f(\alpha_j) =0$. If this is the case, determine whether $g(\alpha_j) > 0$ or $g(\alpha_j)< 0$. Set
        $v^+:=  \#\{ 1\le j\le k \mid f(\alpha_j) = 0 \wedge g(\alpha_j)>0\}$ and  $v^- := \#\{ 1\le j\le k \mid f(\alpha_j) = 0 \wedge g(\alpha_j)<0\}$.

\item Compute $\TaQ(f, g; \alpha, \beta) := v^+- v^- + \sum_{j=0}^k v_j$.

\end{enumerate}

\noindent \hrulefill

\bigskip

\emph{Complexity analysis:} Let  $\delta_X:= \deg_X(\Phi)$, $\delta_Y:= \deg_Y(\Phi)$.

\begin{description}

\item[Step 1.] In a first step, we compute the degrees $D_Y = \deg_Y (\widetilde F G) \le 2d+\delta_Y -1$ and $d_Y = \deg_Y (F) \le d$. Noticing that $D_X=\deg_X (\widetilde F G) \le 2d+\delta_X$ and $d_X = \deg_X F \le d$, we have that all the subresultant polynomials associated with $\widetilde F G$ and $F$ can be computed, by means of the procedure described in Section \ref{sec:complexity}, within complexity
 $$O(d((d+\delta_X+\delta_Y) (d + \delta_X) (d + \delta_Y) ^{\omega + 1} + d M(d(d+\delta_X + \delta_Y)) \log(d(d+\delta_X + \delta_Y)))).$$
From these subresultants, we obtain the polynomials $R_i\in \Z[X,Y]$, $\tau_i$ and $\rho_i\in \Z[X]$ introduced in Notation \ref{not:subres} with no change in the complexity order.

\item[Steps 2 and 3(a).]
 Consider the polynomial
\begin{equation}\label{poliL}
L(X) = \prod_{-1\le i\le N} \ \tau_i \prod_{1\le i \le N+1} \rho_i.
\end{equation}
In order to determine the Thom encodings of the roots of $L$ in the interval $(\alpha,\beta)$  (Step 2) and the signs of the polynomials $(\tau_{i})_{-1\le i \le N}$ and $(\rho_{i})_{1\le i \le N+1}$  between two consecutive roots of $L$ (Step 3(a)), we compute the realizable sign conditions on the family
\[ \text{Der}(L),X-\alpha, \beta-X,  \text{Der}(\tau_{i})_{-1\le i \le N}, \text{Der}(\rho_{i})_{1\le i \le N+1}\]
(see Section \ref{sec:complexity}).

Each of the factors in \eqref{poliL}  has degree bounded by  $d (4d+\delta_X+\delta_Y-1)$; therefore, the degree of $L$ is bounded by $(2N+3)d (4d+\delta_X+\delta_Y-1) = O(d^2(d+\delta_X+\delta_Y))$ and the complexity of computing the realizable sign conditions  is $O(d^6(d+\delta_X + \delta_Y)^3 \log^3(d^2(d+\delta_X + \delta_Y)))$. The remaining computations of these steps do  not modify this complexity order.

The overall complexity of Steps 1 -- 3(a) is of order $$O(d^2(d+\delta_X + \delta_Y)^3 (d^4 \log^3(d+\delta_X + \delta_Y)+ (d+\delta_Y)^{\omega -1})).$$

\item[Steps 3(b) and 4.] These steps require the determination of the sign of Pfaffian functions of the type  $P(x, \varphi(x))$, with $P\in \Z[X,Y] $, at real algebraic numbers given by their Thom encodings (more precisely, at the real roots  $\alpha_j$ of $L$ lying on $(\alpha, \beta)$ and at the endpoints $\alpha$ and $\beta$ of the given interval). We assume an oracle is given to achieve this task.

    At Step 3(b), we use the oracle for Pfaffian functions defined by polynomials with degrees in $X$ bounded by $d( 4d + \delta_X +\delta_Y-1)$ and degrees in $Y$ bounded by $d$ and their derivatives. Taking into account that the multiplicity of a zero of such a function is at most $4d (d+\delta_X+\delta_Y -1) (2d+\delta_Y -1) + (\delta_X +1) d$ (see \cite[Lemma 15]{BJS16}), it follows that the determination of the signs  $\sg(R_{ i}, \alpha_j^+)$ (and $\sg(R_{ i}, \alpha_{j+1}^-)$) for every $i, j$  requires at most $O(d^4 (d+\delta_X+\delta_Y)^2 (d+\delta_Y) )$ calls to the oracle for functions defined by polynomials with degrees in $X$ bounded by $O(\delta_X d (d +\delta_Y)(d+\delta_X+\delta_Y))$ and degrees in $Y$ bounded by $O(d+(\delta_Y-1)d(d +\delta_Y)(d+\delta_X+\delta_Y))$.

 At Step 4, we need  $ O(d^2 (d+\delta_X+\delta_Y))$ calls to the oracle for the Pfaffian functions defined by the polynomials $F$ and $G$, having degrees bounded by $d$.
\end{description}

\begin{remark}
If $\deg_Y(F) = 0 $ or $\deg_Y(G) =0$, the Tarski-query $\TaQ(f,g; \alpha, \beta)$ can be easily computed. If both degrees are zero, it suffices to apply an algorithm for univariate real polynomials (\cite{BPR}). If $\deg_Y(F)=0$ and $\deg_Y(G)>0$, use the oracle to determine the signs that $g$ takes at the zeros of $f(x)=F(x)$. Finally, if $\deg_Y(G) =0$ and $\deg_Y(F) >0$, we can simply replace $G$ with $(Y^2+1)G$.
\end{remark}

Summarizing we have:

\begin{proposition}\label{prop:complexityTQ}
Let $\varphi$ be a function satisfying a differential equation $\varphi'(x) = \Phi(x, \varphi(x))$ where $\Phi \in \Z[X,Y]$ with $\delta_X:= \deg_X(\Phi)$, $\delta_Y :=\deg_Y(\Phi) >0$,  and let $\alpha < \beta$ be rational numbers such that  $[\alpha,\beta]\subset\text{Dom}(\varphi)$.  Given
$f(x) = F(x, \varphi(x))$ and $g(x) = G(x, \varphi(x))$ where $F, G\in \Z[X, Y]$ with $\deg (F) \le d$, $\deg (G) \le d$, Algorithm \texttt{Tarski-query} computes $\TaQ(f, g; \alpha, \beta)$ within $O(d^2(d+\delta_X + \delta_Y)^3 (d^4 \log^3(d+\delta_X + \delta_Y)+ (d+\delta_Y)^{\omega -1})) $ arithmetic operations and comparisons and using $ O(d^4 (d+\delta_X+\delta_Y)^2 (d+\delta_Y) )
$ calls to an oracle for determining the signs of Pfaffian functions associated to $\varphi$ at real algebraic numbers.
\end{proposition}

\section{Decision problem}\label{sec:decision}

This section focuses on the decision problem for formulas involving univariate Pfaffian functions associated to a fixed function $\varphi$ satisfying $\varphi'(x) = \Phi(x, \varphi(x))$ for a polynomial $\Phi \in \Z[X,Y]$ with $\deg_Y(\Phi) >0$.
We will present a symbolic procedure solving this problem  and estimate its complexity, provided an oracle is given for determining signs of functions of this class evaluated at real algebraic numbers.

\bigskip
We first describe a symbolic procedure to determine all the feasible sign conditions on a finite family of Pfaffian functions associated to $\varphi$ in an interval $[\alpha, \beta]\subseteq \mbox{Dom}(\varphi)$, for $\alpha$ and $\beta$ algebraic numbers, that will be the main sub-routine in our decision procedure.

Given Pfaffian functions $f, g_1,\dots, g_s$ associated with $\varphi$, we start showing how to determine the feasible sign conditions of $g_1, \dots, g_s$ over the set $Z = \{x \in  (\alpha, \beta)\mid f(x)=0\}$. To do so, we  follow the algorithm in  \cite{Perrucci11}, which is based in turn on the ones described in \cite[Chapter 2]{BPR}.

The procedure is recursive and in each step, $i=1, \dots, s$, it computes the feasible sign conditions for $g_1, \dots, g_i$ over $Z$ by means of the computation of suitable Tarski-queries and linear equation system solving.
Let $\delta_X =\deg_X(\Phi) $, $\delta_Y = \deg_Y(\Phi) $, and $d$ is an upper bound for the total degrees of the polynomials defining $f, g_1, \dots, g_s$.  By \cite[Corollary 17]{BJS16}, we have that $\# Z \le (d+1) (2d^2-d) ((\delta_Y+3)d +\delta_X)$ and so, there are at most $O(d^3(d \delta_Y+\delta_X))$ feasible sign conditions at each step.

For $i=1$, there are three possible sign conditions $\{x\in Z\mid g_1(x)>0\}$, $\{x\in Z\mid g_1(x)=0\}$, and $\{x\in Z\mid g_1(x)>0\}$. We determine the cardinalities $c^+, c^0$ and $c^{-}$ of these sets taking into account that:
\begin{eqnarray*}
c^0+c^{+}+c^{-} &=& \TaQ(f,1)\\
    c^{+}-c^{-} &=& \TaQ(f,g_1)\\
    c^{+}+c^{-} &=& \TaQ(f,g_1^2)
\end{eqnarray*}
Here and in the next steps, all Tarski-queries are taken with respect to the interval $(\alpha, \beta)$.

For $i>1$, once step $i-1$ is finished, each of the computed feasible sign conditions gives three possible sign conditions for step $i$. So, following \cite{Perrucci11}, to complete step $i$, we have to compute at most $3 \# Z$ Tarski queries of the type
$\TaQ(f, g_1^{\alpha_1} \dots g_{i}^{\alpha_{i}})$ with $\alpha_j \in \{0,1,2\}$ and to solve a linear system of size at most  $(3  \# Z) \times (3  \# Z)$. As the degrees of the polynomials defining the functions  $g_1^{\alpha_1} \dots g_{i}^{\alpha_{i}}$ are bounded by $2di$, by Proposition \ref{prop:complexityTQ}, the complexity of computing any of these Tarski queries is of order
$$O((sd)^2(sd+\delta_X + \delta_Y)^3 ((sd)^4 \log^3(sd+\delta_X + \delta_Y)+ (sd+\delta_Y)^{\omega -1}))$$
and the computation requires
$O((sd)^4 (sd+\delta_X+\delta_Y)^2 (sd+\delta_Y) )$ calls to the oracle for functions defined by polynomials with degrees in $X$ bounded by $O(\delta_X sd (sd +\delta_Y)(sd+\delta_X+\delta_Y))$ and degrees in $Y$ bounded by $O(sd+(\delta_Y-1)sd(sd +\delta_Y)(sd+\delta_X+\delta_Y))$.
The complexity of solving the linear system is  $O(d^6(d \delta_Y+\delta_X)^2)$. Then, the total complexity of this step if of order
$$ O(s^2d^5 (d \delta_Y+\delta_X) (sd+\delta_X + \delta_Y)^3 ((sd)^4 \log^3(sd+\delta_X + \delta_Y)+ (sd+\delta_Y)^{\omega -1})+ (sd)^4 (sd+\delta_X+\delta_Y)^2 (sd+\delta_Y))=$$
$$= O(s^2d^5 (d \delta_Y+\delta_X) (sd+\delta_X + \delta_Y)^3 ((sd)^4 \log^3(sd+\delta_X + \delta_Y)+ (sd+\delta_Y)^{\omega -1})).$$

Therefore, the overall complexity for determining the feasible sign conditions of $g_1,\dots, g_s$ over $Z$ is of order
$$O(s^3d^5 (d \delta_Y+\delta_X) (sd+\delta_X + \delta_Y)^3 ((sd)^4 \log^3(sd+\delta_X + \delta_Y)+ (sd+\delta_Y)^{\omega -1}))$$
and the procedure requires $O(s^5d^7(d \delta_Y+\delta_X) (sd+\delta_X+\delta_Y)^2 (sd+\delta_Y) )$
calls to the oracle for functions with defining polynomials of degrees in $X$ bounded by $O(\delta_X sd (sd +\delta_Y)(sd+\delta_X+\delta_Y))$ and degrees in $Y$ bounded by $O(sd+(\delta_Y-1)sd(sd +\delta_Y)(sd+\delta_X+\delta_Y))$.

Now, to determine \emph{all} the feasible sign conditions of a family of Pfaffian functions
$g_1, \dots, g_s$ associated with $\varphi$ over the interval, we apply the previous procedure as follows:
\begin{enumerate}
\item For $j = 1, \dots, s$, compute the feasible sign conditions of $g_1, \dots, g_{j-1}, g_{j+1}, \dots, g_s$ over the zero set of $g_j$ in the given interval. Although redundant, this procedure gives all the feasible sign conditions in which one of the functions is equal to zero.
\item Consider the function $f= \left(\prod_{1 \le j \le s} g_j\right)'$ and compute the feasible sign conditions of $g_1, \dots, g_s$ over the zero set of $f$. This gives the feasible sign conditions between two consecutive zeros of the functions $g_j$, $1 \le j \le s$, and, therefore, all the sign conditions over these functions consisting only of inequalities.
\end{enumerate}

Note that, taking into account the upper bounds for the number of zeros of $g_1, \dots, g_s$ and $(\prod_{1\le i \le s} g_i)'$ (see \cite[Corollary 17]{BJS16}), the number of feasible sign conditions of $g_1,\dots, g_s$ in an interval contained in $\Dom(\varphi)$ is bounded by $O(\delta_Y(sd +\delta_X+\delta_Y)^4)$.

In step 2, the polynomial defining $f$ has degree bounded by $sd+\delta_X + \delta_Y -1$.
Then, the number of zeros of  $f$ is of order $O((sd+\delta_X+\delta_Y)^4 \delta_Y)$ and the complexity of computing each of the required Tarski queries is of order  $O((sd+\delta_X + \delta_Y)^9  \log^3(sd+\delta_X + \delta_Y))$. Therefore, the complexity of this step is
$$O(s(sd+\delta_X + \delta_Y)^{13}  \log^3(sd+\delta_X + \delta_Y) \delta_Y) + (sd+\delta_X+\delta_Y)^8 \delta_Y^2)) =$$
$$=O(s(sd+\delta_X + \delta_Y)^{13}  \log^3(sd+\delta_X + \delta_Y) \delta_Y ) $$
and requires $O(s (sd+\delta_X+\delta_Y)^{11} \delta_Y)$ calls to the oracle for functions defined by polynomials of degrees in $X$ bounded by $O(\delta_X s^3 (sd+\delta_X+\delta_Y)^3)$ and degrees in $Y$ bounded by $O(s(sd+\delta_X+\delta_Y)+(\delta_Y-1)s^3(sd+\delta_X+\delta_Y)^3)$.

Finally, we call the oracle to compute the signs  of $g_1, \dots, g_s$ in $\alpha$ and in $\beta$.
This finishes the computation of all the feasible sign conditions of the functions
$g_1, \dots, g_s$ in $[\alpha, \beta]$ within complexity of order
$O(s\delta_Y(sd+\delta_X + \delta_Y)^{13}  \log^3(sd+\delta_X + \delta_Y))$ and with
$O(s (sd+\delta_X+\delta_Y)^{11} \delta_Y)$ calls to the oracle for functions defined by polynomials with degrees in $X$ bounded by $O(\delta_X s^3 (sd+\delta_X+\delta_Y)^3)$ and degrees in $Y$ bounded by $O(s(sd+\delta_X+\delta_Y)+(\delta_Y-1)s^3(sd+\delta_X+\delta_Y)^3)$.
	
Therefore, we have proved the following:

\begin{theorem} \label{thm:signconditions} Let  $\varphi$ be a Pfaffian function satisfying $\varphi'(x) = \Phi(x, \varphi (x))$ for $\Phi \in \Z[X,Y]$ with $\deg_Y(\Phi)> 0$. Let $g_1, \dots, g_s$ be functions defined by  $g_i(x) = G_i(x, \varphi(x))$ for $G_i \in \Z[X,Y]$. If $\deg G_i \le d$ for $1\le i \le s$, $\deg_X (\Phi) = \delta_X$ and $\deg_Y (\Phi) = \delta_Y$, then, all the feasible sign conditions for  $g_1, \dots,g_s$ in an interval $[\alpha,\beta] \subseteq \hbox{Dom}(\varphi)$, for $\alpha, \beta$ real algebraic numbers, can be determined by means of a symbolic procedure with complexity $O(s\delta_Y(sd+\delta_X + \delta_Y)^{13}  \log^3(sd+\delta_X + \delta_Y)  )  $ and with  $O(s (sd+\delta_X+\delta_Y)^{11} \delta_Y)$ calls to an oracle for determining signs of functions associated with $\varphi$ defined by polynomials in $\Z[X,Y]$ of degrees in $X$ bounded by $O(\delta_X s^3 (sd+\delta_X+\delta_Y)^3)$ and degrees in $Y$ bounded by $O(s(sd+\delta_X+\delta_Y)+(\delta_Y-1)s^3(sd+\delta_X+\delta_Y)^3)$.
\end{theorem}

As a consequence of Theorem \ref{thm:signconditions} we can establish a complexity result for the decision problem under consideration which implies Theorem \ref{thm:Pfaffiandecision} in the Introduction:

\begin{corollary}
Let $\varphi$ be a Pfaffian function satisfying  $\varphi'(x) = \Phi(x, \varphi (x)$ for $\Phi\in \Z[X,Y]$ with $\deg_X(\Phi) = \delta_X$ and $\deg_Y (\Phi) = \delta_Y > 0$.
Let $\Psi$ be a quantifier-free formula in a variable  $x$ involving  functions  $g_1, \dots, g_s$ defined by $g_i(x) = G_i(x, \varphi(x))$ for $G_i \in \Z[X,Y]$ with  $\deg G_i \le d$, for $1\le i \le s$. There is a symbolic procedure that determines the truth value of the formula $\exists x \, \Psi$  in an interval $[\alpha, \beta] \subseteq \hbox{Dom}(\varphi)$, for $\alpha<\beta$ real algebraic numbers, within complexity $O(s\delta_Y(sd+\delta_X + \delta_Y)^{13} \log^3(sd+\delta_X + \delta_Y) + \delta_Y (sd +\delta_X+\delta_Y)^4 |\Psi| )$, where $|\Psi|$ denotes the length of $\Psi$.
\end{corollary}

\section{E-polynomials} \label{sec:Epolynomials}

In this section we deal with the particular case of E-polynomials, namely when $\varphi(x) = e^{h(x)}$ for a polynomial $h\in \Z[X]$ of positive degree. In this case, our algorithms will run without need of an oracle.

First, we will show how the procedures described in Section  \ref{sec:decision} can be turned into standard symbolic algorithms (not relying on oracles) by using a subroutine for determining the sign of an E-polynomial at a real algebraic number given by its Thom encoding (see \cite[Section 5.1]{BJS16}). Then, we will apply our techniques to solve algorithmically the decision problem for a particular class of formulas involving multivariate E-polynomials. Finally, we will introduce a suitable notion of Thom encoding for the zeros of a univariate E-polynomial and analyze the complexity of its computation.

\subsection{Decision problem}\label{sec:epolydecision}

When dealing with arbitrary Pfaffian functions of the class introduced in Section \ref{subsec:functionclass}, in order to determine the sign that the function takes at a real algebraic number we rely on an oracle.
For E-polynomials, these signs and, consequently, the signs to the left or to the right of an algebraic number that are required for the computation of Tarski-queries, can be computed explicitly. Here we estimate the complexity of these computations.

\begin{proposition}\label{prop:signdetermination}
Let $h\in \Z[X]$ be a polynomial of degree $\delta>0$ and, for $P \in \Z[X,Y]$ with $\deg_X(P) = d_X$ and $\deg_Y(P)= d_Y$, let $p(x) =P(x, e^{h(x)})$. Let $\alpha\in \R$ be a root of a polynomial $L \in \Z[X]$ with $\deg(L) = \ell$ given by its Thom encoding. Then, we can determine the signs $\sg(p, \alpha^+)$ and $\sg(p, \alpha^-)$ within complexity $O(\ell^4 d_Y^5 (d_X+\delta) (\mathcal{H} + (\ell d_Y)^6
(\ell d_Y+ \log (\mathcal{H}))^2))$, where $\mathcal{H}= (\ell + d_X+ d_Y (2d_X+\delta)(\delta -1))! H(L)^{d_X +d_Y (2d_X+\delta)(\delta -1)}((d_Y+1)H(P) (d_Y (2d_X+\delta)(\delta -1) +d_X+d_Y \delta^2 H(h))^{ d_Y (2d_X+\delta)})^\ell$,  assuming $d_X\ge \delta $.
\end{proposition}

\begin{proof}{Proof.} To compute the signs $\sg(p, \alpha^+)$ and $\sg(p, \alpha^-)$
we follow Remark \ref{rem:lateralsign}. Then, it suffices to determine the signs $p^{(\nu)}(\alpha)$ for $0\le \nu\le  \mult (\alpha,p)$. Now, by equation \eqref{eq:multiplicity}, we have that $\mult (\alpha,p) \le d_Y (2d_X+\delta)$. Taking into account that $\deg_X(\widetilde P ^{(\nu)})\le d_X + \nu (\delta -1)$, $\deg_Y (\widetilde P^{(\nu)})\le d_Y$ and $H(\widetilde P ^{(\nu)})\le H(P) \prod_{j=0}^{\nu-1}(j(\delta -1) + d_X + d_Y \delta^2 H(h))$ (see \cite[Remark 19]{BJS16}), this can be done  within the stated complexity by applying the algorithm \texttt{E-SignDetermination} from \cite{BJS16}.
\end{proof}

From the results in Section \ref{subsec:TQalgorithm}, replacing calls to the oracle by the above sign determination, we deduce the following complexity result for the computation of Tarski-queries for E-polynomials.

\begin{proposition}\label{prop:TQepoly}
Let $f(x) = F(x, e^{h(x)})$ and $g(x)= G(x,e^{h(x)})$  be $E$-polynomials defined by $F, G\in \Z[X,Y]$ and $h\in \Z[X]$ with
$\deg(F), \deg(G), \deg(h) \le d$ and $H(F), H(G), H(h) \le H$, and let $(\alpha, \beta)$ be an  interval, where $\alpha, \beta$ are real algebraic numbers given by their Thom encodings, $-\infty$ or $+\infty$.  There is an algorithm that computes $\TaQ(f,g;\alpha,\beta)$  within complexity $(2dH)^{O(d^6)}$.
\end{proposition}

\begin{proof}{Proof.} First, we consider the case $\alpha, \beta \in \R$. The algorithm is based on Algorithm  \texttt{Tarski-query} but each call to the oracle will be replaced by a call to the algorithm of Proposition \ref{prop:signdetermination}. To be able to compute the overall complexity of this algorithm, we will bound the heights of the polynomials involved.
Note that, under our hypothesis, $\deg_X(\widetilde F G) \le 3d-1$ and $\deg_Y(\widetilde F G) \le 2d$ and, therefore,
$\deg_X (R_i) \le 5d^2-d$ y $\deg_Y (R_i) \le d$ para $0 \le i \le N$.

As $H(\widetilde F) \le H d (1+d^2H)$ and $\deg(\widetilde F) \le 2d-1$, Equation \eqref{eq:productheight} implies that  $H(\widetilde F G) \le 4d^3 H^2 (1+d^2H) (d+1)^2 \le 32 d^7 H^3$. Then, using the determinantal formula for the resultants analyzing their expansions along the last column, we have that
$H(R_i) \le (3d)! H^{2d} (32 d^7 H^3)^d
 (d+1)^{2d} (3d)^d \le  2^{5d}3^{4d} d^{13d} H^{5d}$, since $(3d)! (d+1)^{2d} \le 3^{3d}d^{5d}$.

For $0 \le i \le N$,  $\deg(\rho_i)$ and  $\deg(\tau_i)$ are less than or equal to $5d^2-d$ and their heights are at most $ 2^{5d}3^{4d} d^{13d} H^{5d}$;  then, $\deg(L) \le (2d+3)(5d^2-d)$ and
$H(L) \le   (2^{5d}3^{4d} d^{13d} H^{5d})^{2d+3}(5d^2-d+1)^{2d+3} \le (5\cdot 2^{5d}3^{4d} d^{13d+2} H^{5d})^{2d+3}$.

Therefore, computing each of the signs $sg(R_i(x,e^{h(x)}, \alpha_j^+)$ or $sg(R_i(x,e^{h(x)}, \alpha_j^-)$ following Proposition \ref{prop:signdetermination}, costs
$O(d^{12}d^5 d^2 (\mathcal{H}+d^{24}(d^4+\log (\mathcal{H}))^2))= O(d^{19}(\mathcal{H}+d^{24}(d^4+\log (\mathcal{H}))^2)) $ where
$\mathcal{H} \le ((2d+3) (5d^2-d)+5d^2-d+d(10d^2-d)(d-1))! (5\cdot 2^{5d}3^{4d} d^{13d+2} H^{5d})^{(2d+3)(10d^4-5d^3)} ((d+1) 2^{5d}3^{4d} d^{13d} H^{5d}(10d^4-5d^3+d^3 H)^{10d^3-d^2})^{(2d+3)(5d^2-d)} = (2dH)^{O(d^6)}$. Note that computing all these signs and the signs of  $f(\alpha_j)$ and $g(\alpha_j)$ does not increase the order of complexity.

In case $\alpha = -\infty$ or $\beta = +\infty$, we also adapt Algorithm \texttt{Tarski-query}, taking into account that non-bounded intervals appear and, therefore, we have to determine the signs that a sequence of E-polynomials take at $-\infty $  or $+\infty$. This can be easily done as stated in \cite[Section 5.2]{BJS16}.
\end{proof}

As explained in Section \ref{sec:decision}, from the complexity of computing Tarski-queries, we can deduce the complexity for the determination of all the feasible sign conditions on a finite family of E-polynomials:

\begin{proposition}\label{prop:signconditiosexp}
Let $g_1, \dots, g_s$ be $E$-polynomials defined by $g_i(x) = G_i(x, e^{h(x)})$ where $G_i\in \Z[X,Y]$, for $1 \le i \le s$, and $h\in \Z[X]$ are polynomials of degrees bounded by $d$ and heights bounded by $H$. There exists an algorithm that determines all the feasible sign conditions for $g_1, \dots, g_s$ within complexity $(2dH)^{O(s^7d^6)}$.
\end{proposition}

\begin{proof}{Proof.}
It suffices to determine the feasible sign conditions for $g_1, \dots, g_s$ over the zero sets of $f := \prod_{i=1}^{s} g_i$ and $f'$, and to decide the signs of $g_1, \dots, g_s$ in $-\infty$ and in  $+\infty$.

Let $F (X,Y)=  \prod_{i=1}^{s} G_i(X,Y)$. Then $\deg (F) \le sd$  and $H(F) \le H^s (d+1)^{2s}$, and therefore, $\deg(\widetilde{F}) \le sd+d$ and
$H (\widetilde{F}) \le H^s (d+1)^{2s}sd (1+d^2H)\le 2H^{s+1} (d+1)^{2s}sd^3$.
The number of zeros of  $f$ and the number of zeros of  $f'$ are of order $O(s^4d^4)$.

By Proposition  \ref{prop:TQepoly}, the complexity of the computation of each Tarski query of the type  $\TaQ(f, g_1^{\alpha_1} \dots g_i^{\alpha_i})$ with $\alpha_j \in \{0,1,2\}$ or of the type $\TaQ(f', g_1^{\alpha_1} \dots g_i^{\alpha_i})$ with $\alpha_j \in \{0,1,2\}$ is of order  $(2(2sd)(2sd^3H^{s+1}(d+1)^{2s})^{O((sd)^6)} = (2dH)^{O(s^7 d^6)}$.

Note that the number of  Tarski queries to be computed, the solving of the associated linear systems  and the computations of the signs in $-\infty$ and in  $+\infty$ do not modify the order of complexity.
\end{proof}

Applying the previous proposition, we can also solve a decision problem algorithmically in the case of univariate E-polynomials but in this case without need of an oracle:

\begin{theorem}\label{thm:expdecision}
 Let $\Psi$ be a quantifier-free formula  in a variable  $x$ involving $E$-polynomials $g_1, \dots, g_s$ defined by $g_i(x) = G_i(x, e^{h(x)})$ with $G_i \in \Z[X,Y]$, for $1\le i \le s$, and $h\in \Z[X]$ polynomials with degrees bounded by $d$ and heights bounded by $H$. There is a symbolic algorithm that determines whether the formula $\exists x\, \Psi$ is true or false within complexity  $ (2dH)^{O(s^7d^6)}+ O( s^4d^4  |\Psi|)   $, where $|\Psi|$ denotes the length of $\Psi$.
\end{theorem}

\subsection{A decision problem for multivariate E-polynomials}

Our methods can be extended to the case of multivariate E-polynomials, namely Pfaffian functions of the form $f(x_1,\dots, x_n)=F(x_1,\dots, x_n,e^{h(x_1,\dots, x_n)})$ where $F\in \Z[X_1,\dots, X_n,Y]$ and $h \in \Z [X_1, \dots, X_n]$.

First we address the consistency problem for these functions. For $F_1, \dots, F_s \in \Z[X_1,\dots, X_n,Y]$ and $h \in \Z [X_1, \dots, X_n]$, consider the formula
$$\exists \mathbf{x}: F_1(\mathbf{x}, e^{h(\mathbf{x})}) \epsilon_1 0 \land \dots \land F_s(\mathbf{x}, e^{h(\mathbf{x})}) \epsilon_s 0$$
with $\epsilon_i \in \{<, >, =\}$ for $1 \le i \le s$, and $\mathbf{x}= (x_1,\dots,x_n)$.
This formula is equivalent to
\begin{equation}\label{eq:multivariateformula}
\exists z \exists \mathbf{x}: F_1(\mathbf{x}, e^z) \epsilon_1 0 \land \dots \land F_s(\mathbf{x}, e^z)  \epsilon_s 0 \land z= h(\mathbf{x}).\end{equation}

Consider the polynomial formula
$$ \exists \mathbf{x}: F_1(\mathbf{x}, y) \epsilon_1 0 \land \dots \land F_s(\mathbf{x}, y)  \epsilon_s 0 \land z= h(\mathbf{x}).$$
By means of quantifier elimination over $\R$, this formula is equivalent to a quantifier free formula
$\psi(z,y)$. Therefore, formula \eqref{eq:multivariateformula} is equivalent to $\exists z \psi(z,e^z)$ and, applying Theorem \ref{thm:expdecision}, we can decide whether it is true or false.

With the same arguments, we can deal with the decision problem for multivariate prenex formulas with only one block of quantifiers and obtain complexity bounds for the algorithmic solving of the problem:

\begin{proposition}
 Let $\Psi$ be a quantifier-free formula in the variables $\mathbf{x}=(x_1,\dots, x_n)$ defined by $g_i(\mathbf{x}) = G_i(\mathbf{x}, e^{h(\mathbf{x})})$ for $G_i \in \Z[X_1,\dots,X_n,Y]$, for $1\le i \le s$, and $h\in \Z[X_1,\dots,X_n]$ polynomials with degrees bounded by $d$ and heights bounded by $H$.
There is a symbolic algorithm that determines whether the formula  $\exists x_1 \dots \exists x_n \Psi(x_1,\dots, x_n)$ is true or false within complexity  $ (2dH)^{(sd)^{O(n)}}$.
\end{proposition}

\begin{proof}{Proof.}
The result follows straightforwardly from our previous considerations by applying the complexity bounds in \cite[Theorem 14.22]{BPR} for real elimination and Theorem \ref{thm:expdecision} above.
\end{proof}

\subsection{Thom encoding}

Similarly as in the case of real univariate polynomials, we will now show how to encode the zeros of an E-polynomial in one variable by a suitable form of Thom encodings. To do so, we will use the following notions which were already introduced in  \cite{MW12}.

\begin{definition}
Let $f(x) = F(x, e^{h(x)})$ with $F\in \Z[X,Y]$, $F\ne 0$. The \emph{pseudo-degree} of $f$ is defined as
$$\pdeg(f) = \begin{cases} (\deg_Y(F), \deg_X (F(X,0)) & \hbox{if} \ F(X,0)\ne 0 \\
(\deg_Y(F), 0) & \hbox{if} \  F(X,0)= 0 \end{cases}.$$
The \emph{pseudo-derivative} of $f$ is defined as
$$\pder(f) = \begin{cases} f'(x) & \hbox{if} \  \deg (F(X,0)) > 0 \\
\dfrac{f'(x)}{e^{k h(x)}} & \hbox{if} \   F(X,0) \in \R, f'(x)\ne 0, Y^k \mid \widetilde F(X,Y) \ \hbox{and } \ Y^{k+1} \nmid \widetilde F(X,Y)  \\ 0  & \hbox{if} \ f'(x) = 0 \end{cases}.$$
\end{definition}

Given an E-polynomial $f$, for every $i\in \mathbb{N}$, we denote $\pder^{(i)}(f)$ the $i$th successive pseudo-derivative of $f$, that is, $\pder^{(i)}(f) =\pder(\pder^{(i-1)}(f))$.

Note that $\pdeg(\pder(f)) <_{\hbox{lex}} \pdeg(f)$ if $f \notin \R$ (here, $<_{\hbox{lex}}$ is the lexicographic order).
Then, we have that  $\{\pder^{(i)} (f)\}_{i \in \N}$ is a finite family. We will encode the zeros of $f$ by means of the signs of the functions in this family, relying on the following result:

\begin{proposition} Let $f_1, \dots, f_s$ be a family of $E$-polynomials closed under pseudo-derivation (that is, for every $1\le i \le s$, $\pder(f_i) = f_j$ for some $1\le j \le s$). Let $\varepsilon: \{1,\dots, s\} \to \{-1,0,1\}$. Then
$A_{\varepsilon} = \{ x \in \R  \mid \sg(f_i(x)) = \varepsilon(i) \  \hbox{for every} \ 1\le i \le s \}$ is either empty or a point or an open interval.
\end{proposition}
\begin{proof}{Proof.} By induction on $s$. If $s=1$, $f_1=0$ and there is nothing to prove.
Let $\{f_1, \dots, f_s, f_{s+1}\}$ be a family of E-polynomials closed under pseudo-derivation and assume that $f_{s+1}$ has the maximum pseudo-degree. Then,
$\{f_1, \dots, f_s\}$ is closed under pseudo-derivation and so, by inductive assumption, $ A_\varepsilon= \{ x \in \R  \mid \sg(f_i(x)) = \varepsilon(i) \  \hbox{for} \ 1\le i \le s \}$ is either empty or a point or an open interval. The only case to consider is when $A_\varepsilon$ is an interval.  As $\pder(f_{s+1}) \in \{f_1, \dots, f_s\}$, $\pder(f_{s+1})$ has constant  sign over $A_\varepsilon$.
If $\pder(f_{s+1}) =0$, then $f_{s+1}$ is a constant function. If $\pder(f_{s+1}) \ne 0$, as $\sg(\pder(f_{s+1})) = \sg (f'_{s+1})$, then $ f_{s+1}$ is a strictly monotonous function in  $A_\varepsilon$. The result follows.
\end{proof}

\begin{corollary}\label{cor:ThomEnc}
The zeros of an E-polynomial $f$ are uniquely determined by the feasible sign conditions of $(\pder^{(i)} (f))_{1 \le i \le D}$ over $\{x \in \R \mid f(x)=0\}$, where
$D = \min \{i \mid \pder^{(i+1)} (f) = 0 \}.$
\end{corollary}

This result allows us to define a notion of \emph{Thom encoding} for zeros of  E-polynomials:

\begin{definition} Let $f(x) = F(x, e^{h(x)})$ be an E-polynomial and $D = \min \{i \mid \pder^{(i+1)} (f) = 0 \}$. Let $\varepsilon: \{0,\dots, D\} \to \{-1,0,1\}$ with $\varepsilon(0)=0$. We say $(f,\varepsilon)$  is a \emph{Thom encoding} of the real number $\xi$ if
$ \{ x \in \R  \mid \sg(\pder^{(i)} (f)(x)) = \varepsilon(i) \  \hbox{for} \ 0\le i \le D \} =\{\xi\} $.
  \end{definition}

A result analogous to \cite[Proposition 2.37]{BPR} holds in our context and allows us to use Thom encodings to order all the real zeroes of a given E-polynomial:

\begin{remark}\label{rem:orderThom}
Let $f$ be an  E-polynomial and $\xi_1, \xi_2$ t
wo different real zeros of $f$. By Corollary
\ref{cor:ThomEnc}, $\xi_1$ and $\xi_2$ have two different Thom encodings
$(f,\varepsilon_1)$ and $(f,\varepsilon_2)$, with $\varepsilon _j : \{0,\dots,D\}\to \{-1,0,1\}$ for $j=1,2$. Let $k=\max\{0\le i \le D \mid \varepsilon_1(i) \ne \varepsilon_2(i) \}$
 (then $ \varepsilon_1(k+1)=\varepsilon_2(k+1) \ne 0$). If $\varepsilon_1(k+1)=\varepsilon_1(k+1) =1$, then $\xi_1 > \xi_2$ if and only if $\varepsilon_1(k)>\varepsilon_2(k)$ and, if $\varepsilon_1(k+1)=\varepsilon_1(k+1) =-1$, then $\xi_1 > \xi_2$ if and only if $\varepsilon_1(k)<\varepsilon_2(k)$.
\end{remark}

The following easy example  illustrates Thom encodings in the E-polynomial setting.
\begin{example}
Let $f(x) = (6x-1)e^{2x}-(8x+1)e^x-1 $. This E-polynomial has three real zeros, $\xi_1,\xi_2, \xi_3$. The sequence of pseudo-derivatives of $f$ is $(\pder^{(i)}(f))_{1\le i \le 4}$, where
\begin{eqnarray*}
\pder^{(1)}(f)(x) &=& (12x+4) e^{x}-(8x+9) e^x,\\
\pder^{(2)}(f)(x) &=& (12x+16) e^{x}-8, \\
\pder^{(3)}(f)(x) &=& 12 x +28, \\
\pder^{(4)}(f)(x) &=& 12.
\end{eqnarray*}
The Thom encodings of $\xi_1,\xi_2, \xi_3$ as zeros of $f$ are given by
$$\varepsilon_1 = (0,-1,-1,-1,1), \ \varepsilon_2 = (0,-1,-1,1,1), \ \varepsilon_3 = (0,1,1,1,1).$$
Comparing them as stated in Remark \ref{rem:orderThom}, it follows that $\xi_1 <\xi_2$ (since $\varepsilon_1(4)=\varepsilon_2(4)=1$ and $\varepsilon_1(3)<\varepsilon_2(3)$) and $\xi_2< \xi_3$ (since $\varepsilon_2(i)=\varepsilon_3(i)=1$ for $i=4,3$ and  $\varepsilon_2(2)<\varepsilon_3(2)$).
\end{example}

Unfortunately, we do not know whether Thom encodings can be used to deal with arithmetic operations between zeros of $E$-polynomials defined from the same polynomial $h\in \Z[X]$, since the result of such operation may not be a zero of a function in the same class.
For example, a question posed in \cite{Richardson91} asks whether the set $L= \{x \in \R \mid F(x,e^x) =0, F \in \Q[X,Y]\}$ is closed under addition. Here we prove that the answer is negative:

Assuming $L$ is closed under addition, as $\ln(2) \in L$ (since it is a zero of $e^x -2$), it follows that $\ln(2) +1 \in L$. Then, there exists a non-zero polynomial $F\in \Q[X,Y]$ such that $F(\ln(2) +1, 2e) = 0$ and, therefore, $e$ is algebraic over $\Q(\ln (2))$. Similarly, $\ln(2) + \sqrt{2} \in L$ and then,
$e^{\sqrt{2}}$ is algebraic  over $\Q(\ln (2)+\sqrt{2})$. As a consequence,
the transcendence degree of $\Q(\sqrt{2}, \ln(2), e, e^{\sqrt{2}})$ is $1$, contradicting the
fact that $\{e,e^{\sqrt{2}}\}$ is algebraically independent over $\Q$ by the  Lindemann-Weierstrass theorem.

\bigskip

Our previous results from Section \ref{sec:epolydecision} allow us to compute Thom encodings for zeros of E-polynomials algoritmically. First, we estimate the length of these encodings.

\begin{lemma} Let $f(x) = F(x, e^{h(x)})$, with $F\ne 0$, and $D = \min \{i \mid \pder^{(i+1)} (f) = 0 \}$. If $\deg(F) = d$ and  $\deg(h)=\delta\ge 2$, then $D \le \delta^{d} (d+1)$   and  the total degree of the polynomial defining  $\pder^{(i)} (f)$ is bounded by  $
\delta^{d} (d+1)$ for $ 1 \le i \le D$. If $\deg (h) = 1$, then $D \le d (d+1)$ and the total degree of the polynomial defining $\pder^{(i)} (f)$ is bounded by $d$ for $ 1 \le i \le D$.
\end{lemma}

\begin{proof}{Proof.} To estimate $D$ we are going to analyze the pseudo-degrees of the successive pseudo-derivatives of $f$ and the degrees in $X$ of the polynomials defining them.

Observe that if an E-polynomial $f(x) = F(x, e^{h(x)})$ has pseudo-degree $\pdeg(f) = (m_0, n_0)$ and  $\deg_X (F) = d_0$, then its pseudo-derivative $\pder(f)$ is defined by a polynomial with degree in $X$ bounded by $d_0+\delta-1$ and  satisfies
$$\pdeg(\pder(f)) = \begin{cases} (m_0, n_0-1) & {\rm if}\  n_0 \ne 0 \\
(m_0', n_0')  & {\rm if}\  n_0= 0
\end{cases}$$
where $m_0' \le m_0-1$ and $n_0' \le d_0+\delta-1$.
If $m_0 = 0$, after $n_0$ pseudo-derivation steps, we get $\pdeg(\pder^{(n_0)}(f)) = (0,0)$ and no further pseudo-derivation is needed.
If $m_0>0$, after $n_0+1$ pseudo-derivations, the first coordinate of the pseudo-degree is smaller than $m_0$. Let $(m_1, n_1) = \pdeg(\pder^{(n_0+1)}(f))$ and $d_1$ the degree in $X$ of the polynomial defining $\pder^{(n_0+1)}(f)$. Then, $m_1 \le m_0-1$, $d_1 \le d_0\delta+\delta-1$ and, therefore, $n_1 \le d_0\delta+\delta-1$.

As a consequence, given $f(x)= F(x, e^{h(x)})$, we obtain the sequence of pseudo-degrees $(m_i, n_i)_{0\le i \le k}$ defined as $(m_0, n_0)= \pdeg(f)$ and, for $i\ge 1$, if $m_{i-1}\ne0$, $N_{i-1} = \sum_{j=0}^{i-1} (n_j+1)$ and  $$(m_i, n_i)= \pdeg(\pder^{(N_{i-1})}(f)).$$ If $m_{i-1} = 0$, then $k=i$,   $N_{i-1} = \sum_{j=0}^{i-1} (n_j+1) -1 $ and $(m_k,n_k) = \pdeg(\pder^{(N_{k-1})}(f)) = (0,0)$. Let $d_i$ be the degree in $X$ of the polynomial defining  $\pder^{(N_{i-1})}(f)$ for $0 \le i \le k$.
By our previous considerations,
$m_{i}\le m_{i-1}-1$, $d_i \le d_{i-1}\delta+\delta -1$ and, therefore, $n_i \le d_{i-1}\delta+\delta -1$  for every $i\ge 1$. In particular, $k\le m_0\le d$ and $D = N_{k-1}$. We have that
$$ N_{k-1} \le  \sum_{j=0}^{k-1} (n_j+1) \le \sum_{j=0}^{k-1} (d_j+1) \le \sum_{j=0}^{k-1} \delta^j (d+1) = \begin{cases} \dfrac{\delta^k -1}{\delta -1} (d + 1)  & {\rm if}\  \delta \ge 2  \\ k(d+1) & {\rm if}\  \delta=1 \\ \end{cases}$$
and the result follows.
\end{proof}

Applying  Proposition \ref{prop:signconditiosexp} and the previous result we deduce the following:

\begin{proposition} Given an $E$-polynomial $f(x) = F(x, e^{h(x)})$, where $F \in \Z[X,Y]$ and $h \in \Z[X]$ are polynomials with degree bounded by $d$ and height bounded by $H$, the complexity of finding the Thom encodings of all the zeros $f$ is of order $(2dH)^{d^{O(d)}}$. In case $\deg (h) = 1$ the complexity is of order $(2dH)^{O(d^{22} )}$.
\end{proposition}

\end{document}